\documentclass[11pt]{article}
\usepackage{amsmath,amsfonts,amssymb,amsthm,enumerate,graphicx}
\usepackage{tikz}
\usepackage{subfigure}
\usepackage{amsmath}
\usepackage{amssymb}
\usepackage[english]{babel}
\usepackage{graphicx}
\usepackage{gastex}
\usepackage{longtable}
\usepackage{lscape}
\usepackage{multicol}
\usepackage{latexsym}
\usepackage{verbatim}
\usepackage{multirow}
\usepackage{longtable}
\usepackage{authblk}
\usepackage{braket}
\usetikzlibrary{arrows}

\usepackage{amssymb,latexsym,amsmath,epsfig,amsthm,mathrsfs}
\usepackage{amsfonts}
\usepackage{amscd}
\usepackage{dsfont}
\usepackage{bbm}
\usepackage{rotating}
\usepackage{newlfont}
\usepackage{enumitem}
\usepackage{tabularx}
\usepackage{longtable}
\usepackage{tabularx,ragged2e,booktabs,caption}
\usepackage{multicol}
\usetikzlibrary{arrows}
\tikzstyle{vertex}=[ draw, inner sep=0pt, minimum size=0pt]

\usepackage{epstopdf}
\usepackage{breakurl}
\usepackage{hyperref}
\usepackage{float}
\usepackage{enumitem}
\usepackage{mathrsfs}
\usepackage[mathscr]{euscript}
\usepackage{mathtools}
\usepackage{tikz-cd}
\usepackage{array,multirow}
\usepackage[]{cite}
\usepackage{ragged2e}
\newlist{subquestion}{enumerate}{1}
\date{}
\newtheorem{theorem}{{\bf Theorem}}[section]

\newtheorem{proposition}[theorem]{{\bf Proposition}}

\numberwithin{Subcase}{Case}

\numberwithin{Subsubcase}{Subcase}

\begin{document}

\title{2-uniform covers of $2$-semiequivelar toroidal maps}

 	\author {Dipendu Maity}
 	\affil{Department of Sciences and Mathematics,
 		Indian Institute of Information Technology Guwahati, Bongora, Assam-781\,015, India.~~
 		Email id : dipendu@iiitg.ac.in}

\date{\today}

\maketitle

\begin{abstract}
If every vertex in a map has one out of two face-cycle types, then the map is said to be $2$-semiequivelar. A 2-uniform tiling is an edge-to-edge tiling of regular polygons having $2$ distinct transitivity classes of vertices. Clearly, a $2$-uniform map is $2$-semiequivelar. The converse of this is not true in general. There are 20 distinct 2-uniform tilings (these are of $14$ different types) on the plane. In this article, we prove that a $2$-semiequivelar toroidal map $K$ has a finite $2$-uniform cover if the universal cover of $K$ is $2$-uniform except of two types.
\end{abstract}

\noindent {\small {\em MSC 2010\,:} 52C20, 52B70, 51M20, 57M60.

\smallskip

}

\section{Introduction}

A map is a connected $2$-dimensional cell complex on a surface. Equivalently, it is a cellular embedding of a connected graph on a surface. In this article, a map will mean a polyhedral map on a surface, that is, non-empty intersection of any two faces is either a vertex or an edge. 

For a map $\mathcal{K}$, let $V(\mathcal{K})$ be the vertex set of $\mathcal{K}$ and $u\in V(\mathcal{K})$. The faces containing  $u $ form a cycle (called the {\em face-cycle} at  $u $)  $C_u $ in the dual graph  of  $\mathcal{K} $. That is,  $C_u $ is of the form  $(F_{1,1}\mbox{-}\cdots\mbox{-}F_{1,n_1})\mbox{-}\cdots\mbox{-}(F_{k,1}\mbox{-}\cdots \mbox{-}F_{k,n_k})\mbox{-}$ $F_{1,1} $, where  $F_{i,\ell} $ is a  $p_i $-gon for  $1\leq \ell \leq n_i $,  $1\leq i \leq k $,  $p_r\neq p_{r+1} $ for  $1\leq r\leq k-1 $ and  {$p_k\neq p_1 $}. In this case, the vertex $u$ is said to be of type $ [p_1^{n_1}, \dots, p_k^{n_k}]$ (addition in the suffix is modulo  $k $). A map  $\mathcal{K} $ is said to be {\em 2-semiequivelar} of type $[p_1^{n_1}, \dots, p_k^{n_k};q_1^{m_1}, \dots, q_k^{m_s}]$ if $V(\mathcal{K}) = V_1 \sqcup V_2$ such that the vertices of $V_1$ is of type $[p_1^{n_1}, \dots, p_k^{n_k}]$ and  the vertices of $V_2$ is of type $[q_1^{m_1}, \dots, q_k^{m_s}]$. So, clearly, if $[p_1^{n_1}, \dots, p_k^{n_k}] = [q_1^{m_1}, \dots, q_k^{m_s}]$  then the map is called  $semiequivelar$ of type $[p_1^{n_1}, \dots, p_k^{n_k}]$. A semiequivelar map is said to be an $equivelar~map$ if it consists of same type of faces.   

A {\em $2$-uniform tiling} is an edge-to-edge tiling of regular polygons having $2$ distinct transitivity classes of vertices. A vertex-transitive map is a map on a closed surface on which the automorphism group acts transitively on the set of vertices. A $2$-uniform tiling or map will have vertices that we could label $X$, and others that we could label $Y$. Each $X$ vertex can be mapped onto every other $X$ vertex, but cannot be mapped to any $Y$ vertex. 
Clearly, an $2$-uniform map is $2$-semiequivelar. 


A {\em semiregular} tiling of $\mathbb{R}^2$ is also known as {\em Archimedean}, or {\em homogeneous}, or {\em uniform} tiling. In \cite{GS1977}, Gr\"{u}nbaum and Shephard showed that there are exactly eleven types of Archimedean tilings on the plane. These types are $[3^6]$, $[3^4,6^1]$, $[3^3,4^2]$,  $[3^2,4^1,3^1,4^1]$, $[3^1,6^1,3^1,6^1]$, $[3^1,4^1,6^1,4^1]$, $[3^1,12^2]$, $[4^4]$, $[4^1,6^1,12^1]$, $[4^1,8^2]$, $[6^3]$.
Clearly, a {\em semiregular} tiling on $\mathbb{R}^2$ gives a semiequivelar map on $\mathbb{R}^2$. But, there are semiequivelar maps on the plane which are not (not isomorphic to) an Archimedean tiling. In fact, there exists $[p^q]$ equivelar maps on $\mathbb{R}^2$ whenever $1/p+1/q<1/2$ (e.g., \cite{CM1957}, \cite{FT1965}). Thus, we have

\begin{proposition} \label{prop:plane}
There are infinitely many types of equivelar maps on the plane $\mathbb{R}^2$.
\end{proposition}

We know that the plane is the universal cover of the torus. Since there are infinitely many equivelar maps on the plane, it is natural to ask that what are the other types of semiequivelar maps exist on the torus.  Here we have the following result. 

\begin{proposition} \cite{DM2017, DM2018} \label{theo:GrSh}
Let $X$ be a semiequivelar map on a surface $M$. If $M$ is the torus then the type of $X$ is $[3^6]$, $[6^3]$, $[4^4]$, $[3^4,6^1]$, $[3^3,4^2]$, $[3^2,4^1,3^1,4^1]$,  $[3^1,6^1,3^1,6^1]$,
$[3^1,4^1,6^1,4^1]$,  $[3^1,12^2]$, $[4^1,8^2]$  or $[4^1,6^1,12^1]$.
\end{proposition}



We know that all the Archimedean tiling are vertex-transitive. But, it not true on the torus. Here, we know the following.   

\begin{proposition} \cite{DM2017, DM2018} \label{prop:36&44}
Let $X$ be an equivelar map on the torus. If the type of $X$ is $[3^6]$, $[4^4]$, $[6^3]$ or $[3^3,4^2]$ then $X$ is vertex-transitive. If the type is $[3^2,4^1,3^1,4^1]$,  $[3^1,6^1,3^1,6^1]$, $[3^1,4^1,6^1,4^1]$,  $[3^1,12^2]$, $[4^1,8^2]$, $[3^4,6^1]$  or $[4^1,6^1,12^1]$ then there exists a semiequivelar toroidal map of which is not vertex-transitive.
\end{proposition}

We know that the $2$-sphere $\mathbb{S}^2$ is simply connected. So,  the boundary of the pseudo-rhombicuboctahedron (which is a semiregular spherical map of type $[4^3, 3]$) has no other cover. We also know that this map is not vertex-transitive. Thus, this map has no vertex-transitive cover. We know that the Archimedean tilings are vertex-transitive. So, each semiequivelar toroidal map has vertex-transitive universal cover. Here, we know 

\begin{proposition} \cite{BD2020} If $X$ is a semiequivelar toroidal map then there exists a covering $\gamma \colon Y \to X
$ where $Y$ is a vertex-transitive toroidal map.
\end{proposition}

The {\em $2$-uniform} tilings of the plane $\mathbb{R}^2$ are the generalization of vertex-transitive tilings on the plane. We know from \cite{GS1977, GS1981, Otto1977} that there are 20 2-uniform tiling of types
\begin{align*}
   & [3^6;3^3;4^{2}], [3^{6};3^2,4^1,3^1,4^1],  [3^4,6^1;3^2,6^{2}], [3^{3}, 4^2;3^1,4^1,6^1,4^1], [3^3, 4^2;3^2,4^1,3^1,4^1],\\
   &  [3^{6}; 3^2, 4^1,12^1], [3^1, 4^1, 6^1, 4^1; 4^1, 6^1, 12^1], [3^2,4^1,3^1,4^1;3^1, 4^1,6^1,4^1], [3^2,6^2; 3^1, 6^1, 3^1, 6^1],\\
   & [3^1, 4^1, 3^1, 12^1; 3^1, 12^2], [3^1,4^2,6^1; 3^1, 4^1, 6^1, 4^1], [3^1,4^2, 6^1; 3^1, 6^1, 3^1, 6^1], [3^3,4^2;4^4], [3^{6};3^4,6^1].
\end{align*}
 on the plane (see in Section \ref{2uniform}). 
Since the plane is the universal cover of the torus, so, these types of maps also exist on the torus. Here we know the following.

\begin{proposition} \cite{MDD2020}
Let $X$ be a $2$-semiequivelar map on the torus that is the quotient of the plane's $2$-uniform lattice. Let the vertices of $X$ form $m$ Aut$(X)$-orbits. Then, $m \leq k$ for some positive integer $k$.
\end{proposition}

In this article, we prove the following.

\begin{theorem}\label{theo1}

\begin{enumerate}
    \item[(a)] If $X$ is a $2$-semiequivelar toroidal map that is the quotient of the plane's $2$-uniform lattice of type other than $[3^3, 4^2;3^2,4^1,3^1,4^1]$ and $[3^1,4^2, 6^1; 3^1, 6^1,$ $3^1, 6^1]$ then there exists a covering $\alpha : Y \to X$ where $Y$ is a $2$-uniform toroidal map.
    \item[(b)] If $X$ is a $2$-semiequivelar toroidal map that is the quotient of the plane's $2$-uniform lattice of type $[3^3, 4^2;3^2,4^1,3^1,4^1]$ or $[3^1,4^2, 6^1; 3^1, 6^1, 3^1, 6^1]$ then there does not exist any covering $\gamma : Y \to X$ where $Y$ is a $2$-uniform toroidal map.
    \item[(c)] If $X$ is a $2$-semiequivelar toroidal map that is not the quotient of the plane's $2$-uniform lattice then there does not exist  any $2$-uniform toroidal map $Y$ such that $\beta : Y \to X$ is a covering map.
\end{enumerate}
\end{theorem}

\section{$2$-uniform tilings of the plane}\label{2uniform}

We first present $20$ $2$-uniform tilings on the plane. These are also given in \cite{MDD2020}. We need these for the proofs of our results in Section \ref{sec:proofs-1}.

\begin{figure}[H]
    \centering
    \includegraphics[height=6cm, width= 6cm]{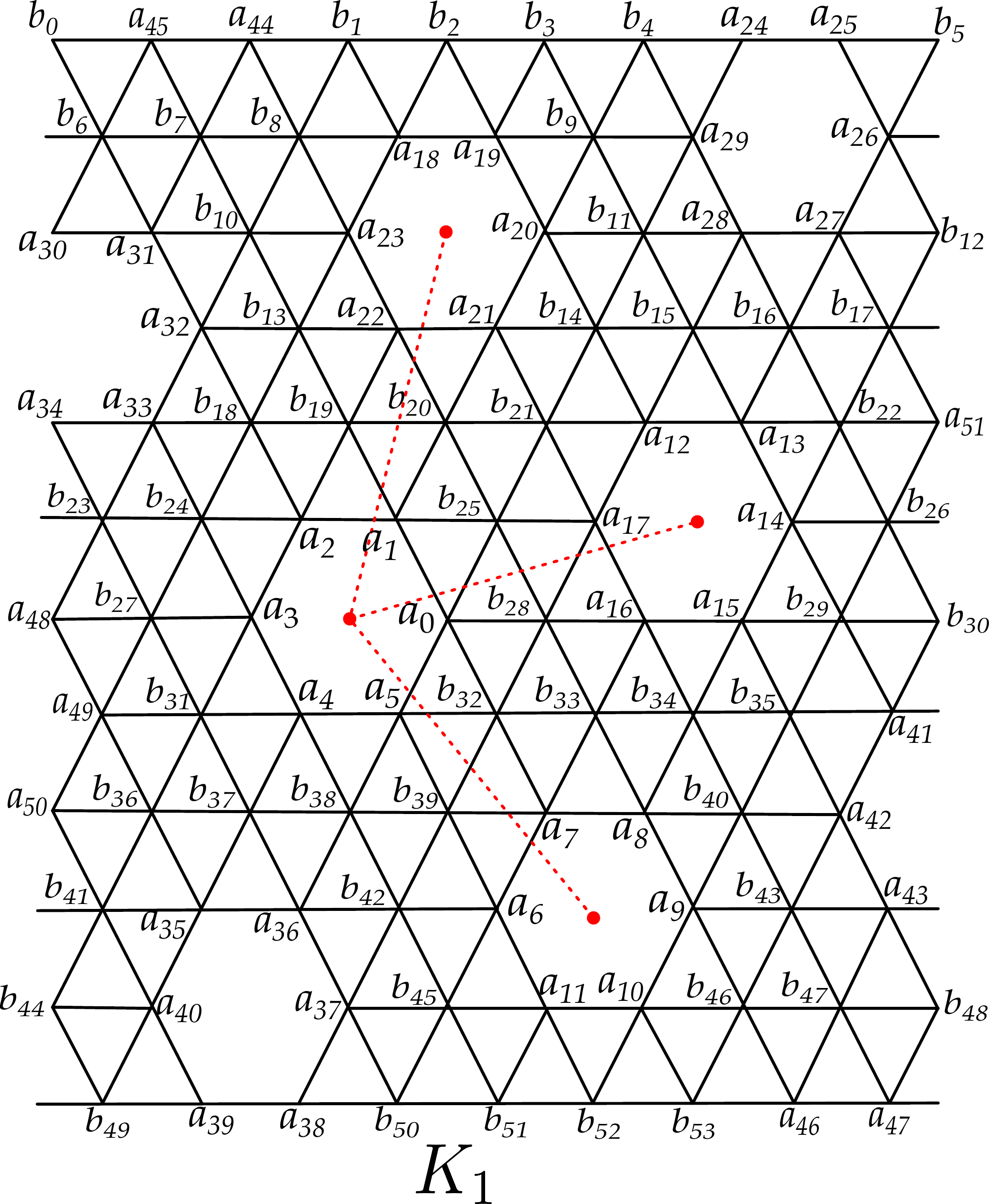}\hspace{5mm}
    \includegraphics[height=6cm, width=6cm]{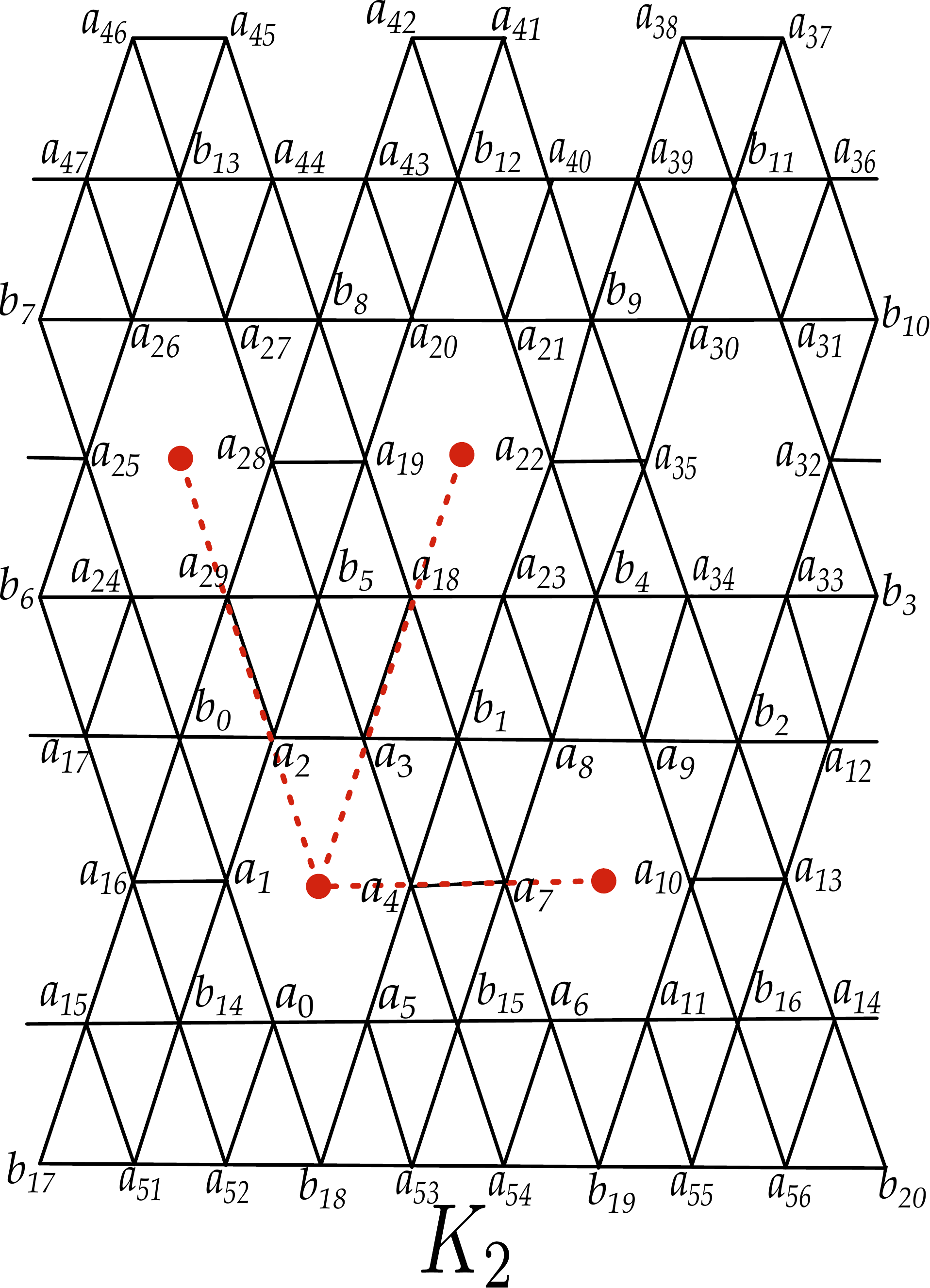}
     \vspace{10mm}
     
    \includegraphics[height=6cm, width= 6cm]{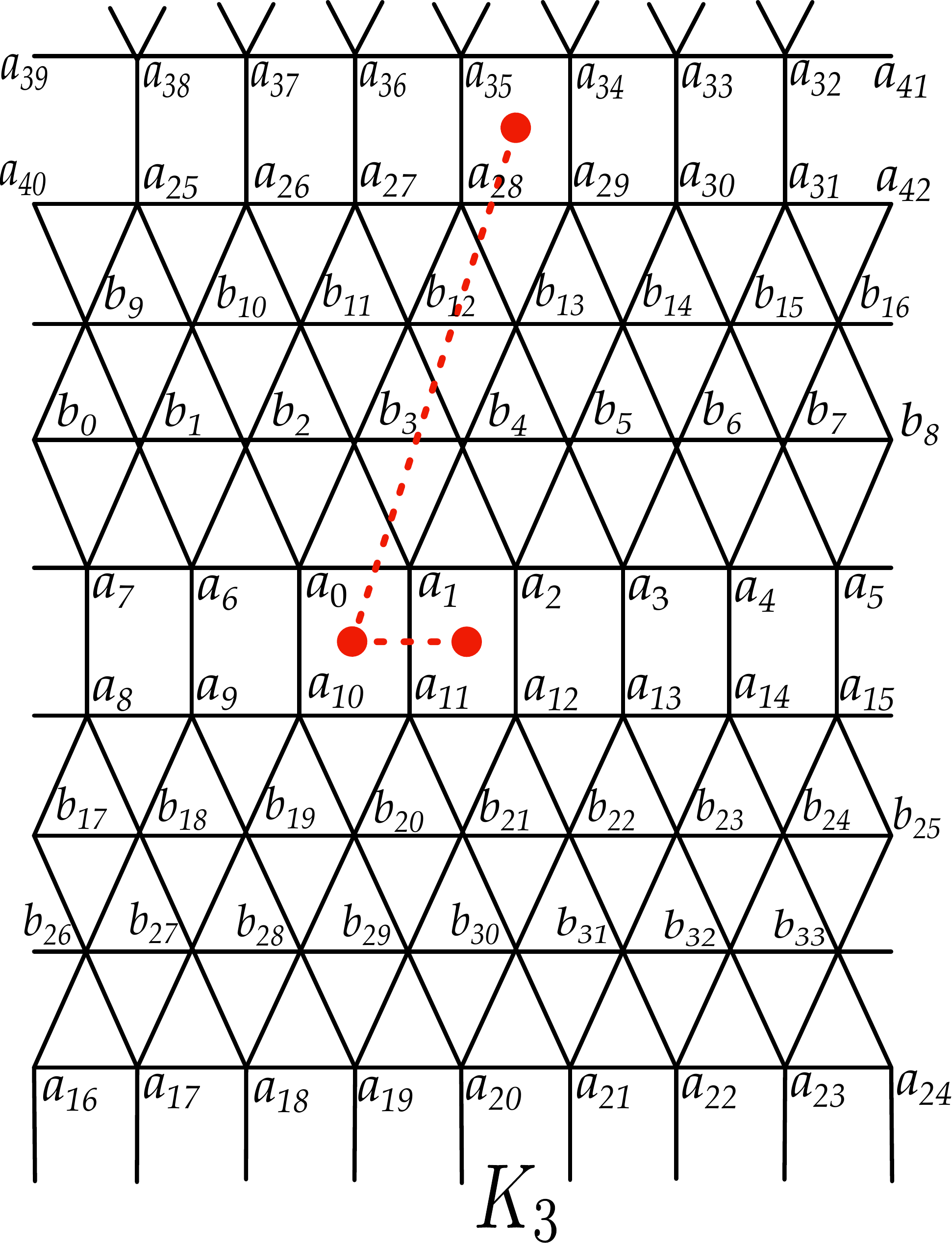}\hspace{5mm}
    \includegraphics[height=6cm, width= 6cm]{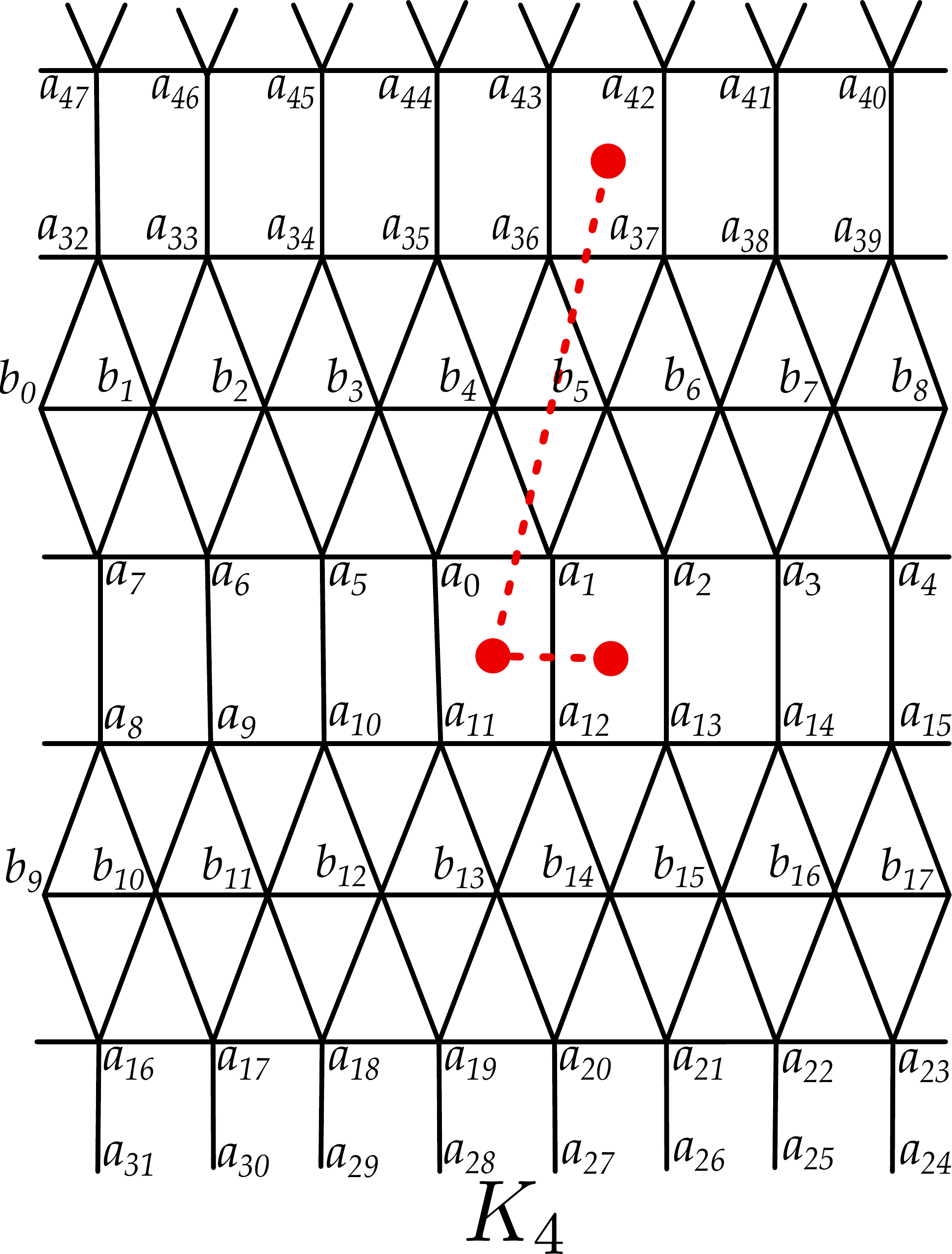}
    \vspace{10mm}
    \end{figure}
    
    \begin{figure}[H]
    \centering
    \includegraphics[height=6cm, width= 6cm]{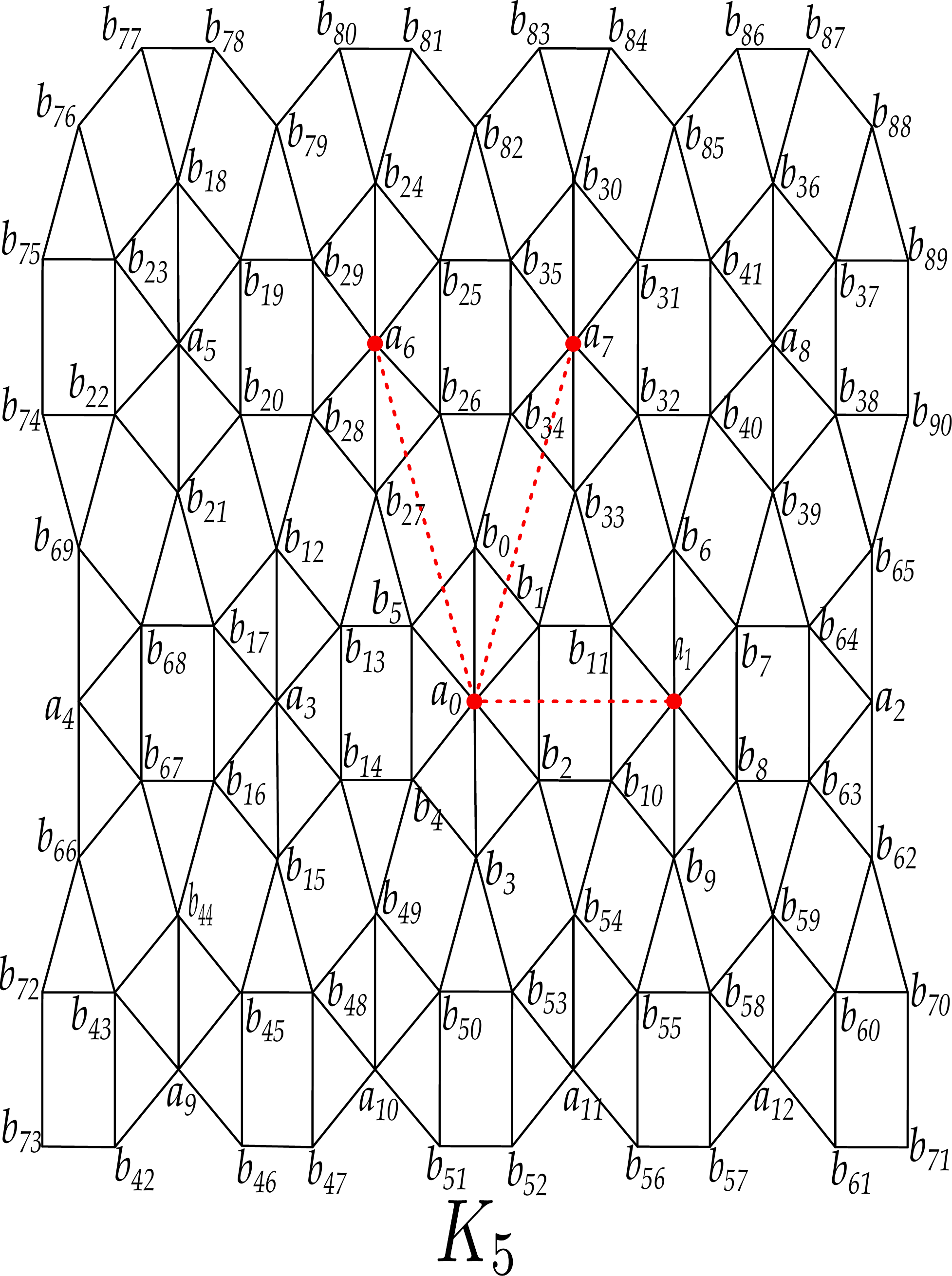}\hspace{5mm}
    \includegraphics[height=6cm, width= 6cm]{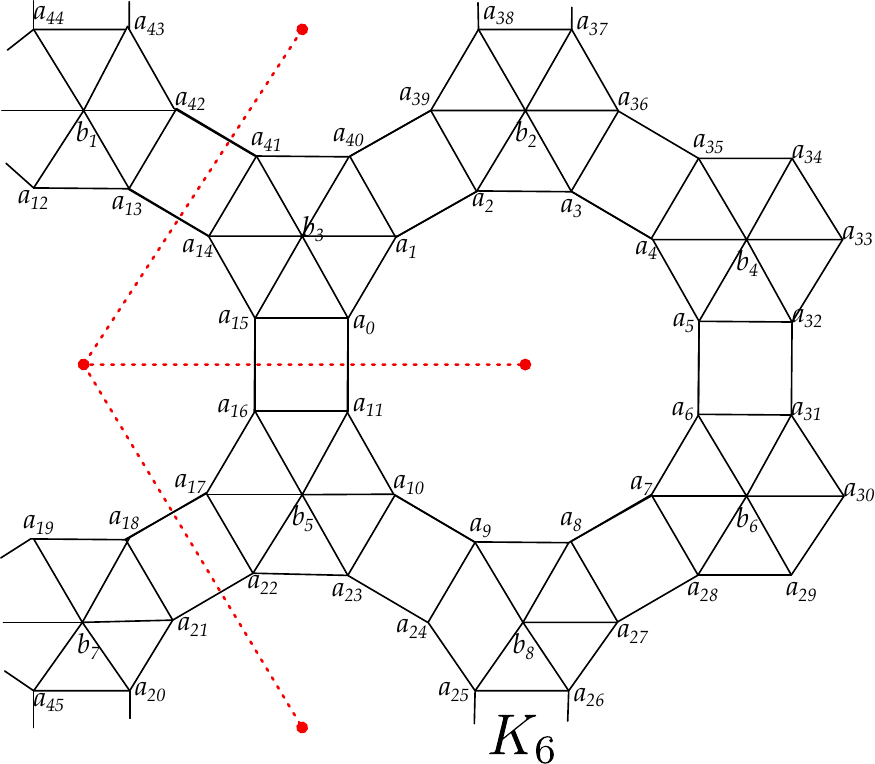}

    \end{figure}
    
    \begin{figure}[H]
    \centering
    \includegraphics[height=6cm, width= 6cm]{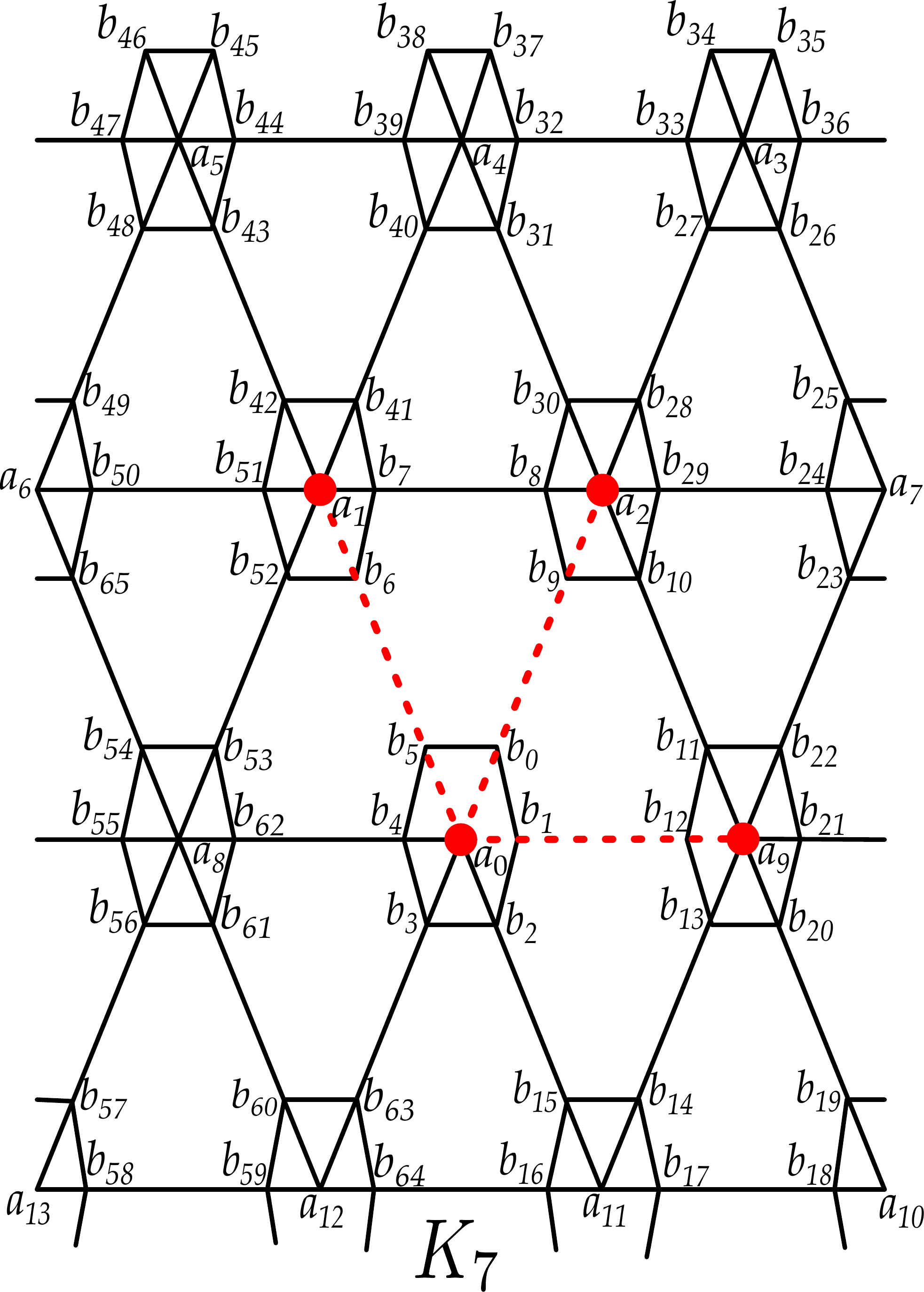}
    \includegraphics[height=6cm, width= 6cm]{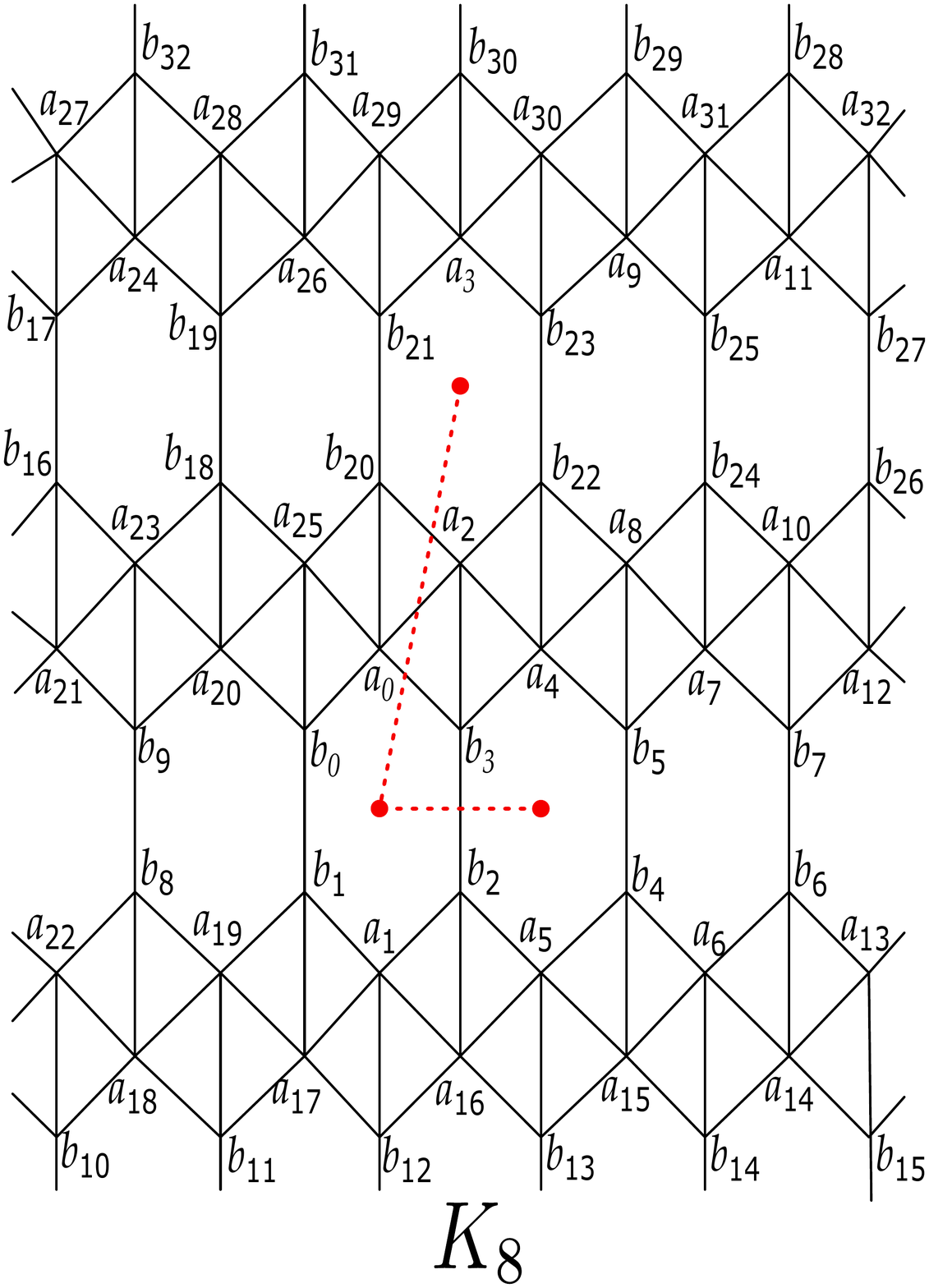}
     \vspace{10mm}
     
    \includegraphics[height=6cm, width= 6cm]{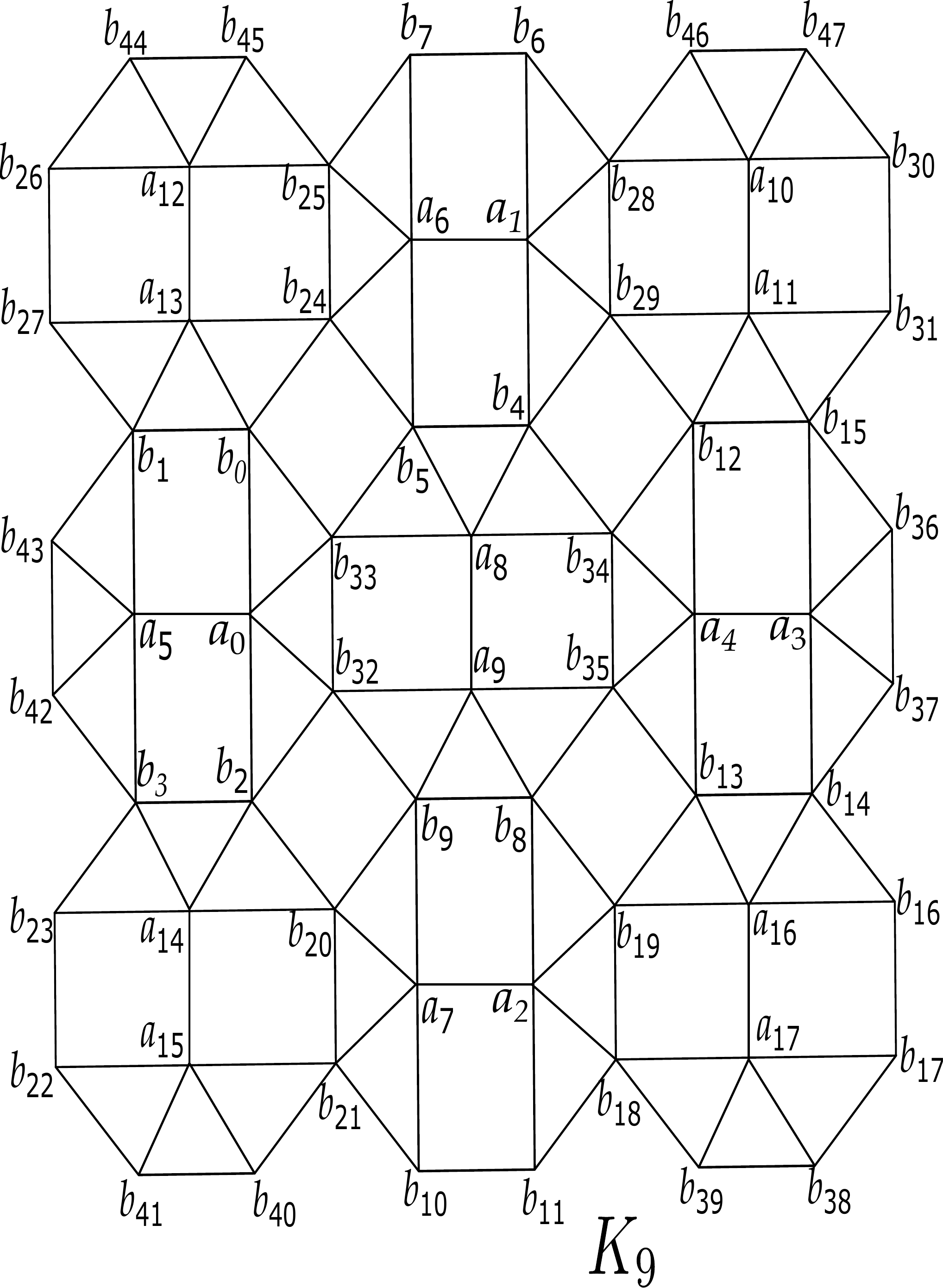}\hspace{5mm}
    \includegraphics[height=6cm, width= 6cm]{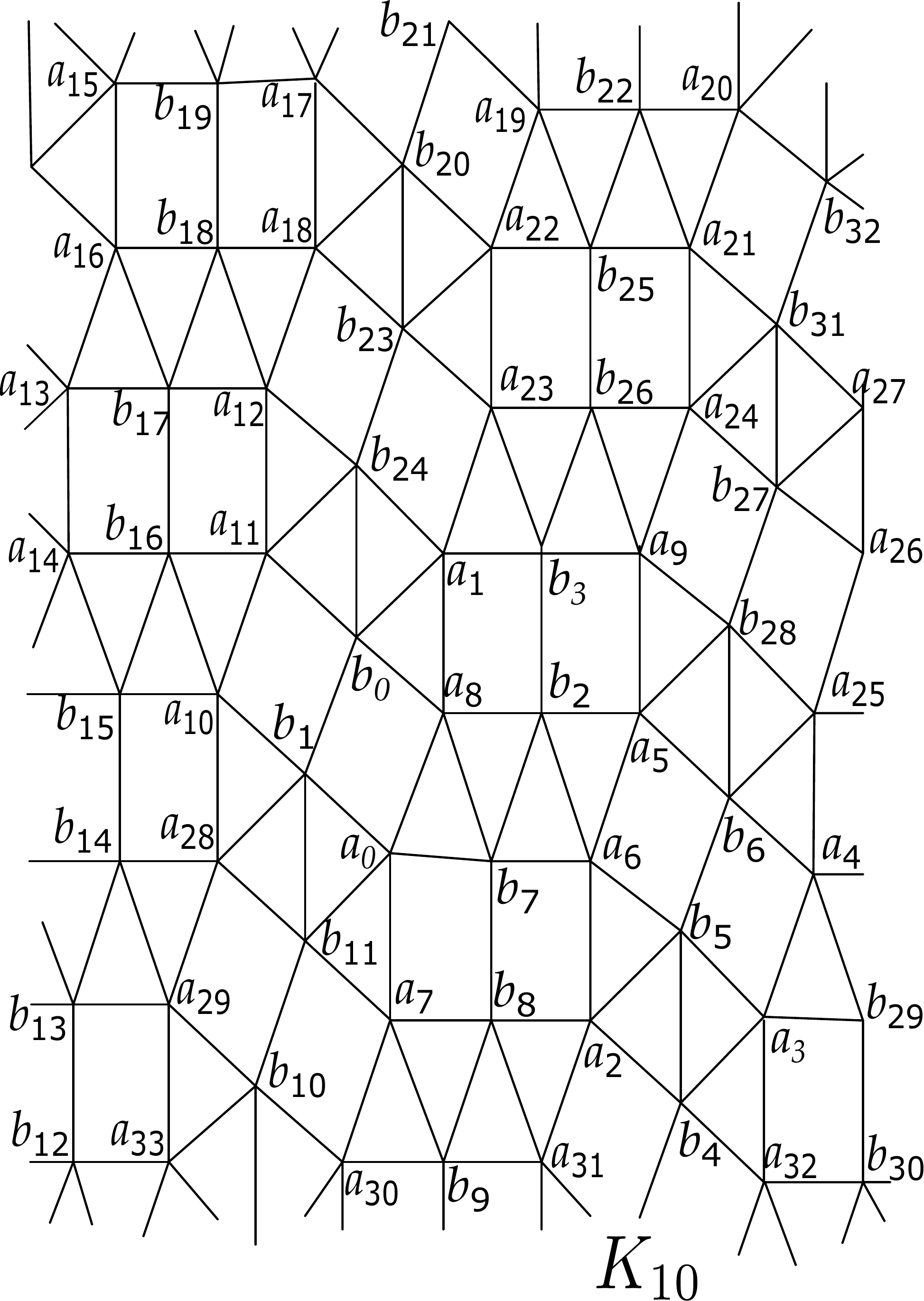}
    
     \vspace{10mm}
     \end{figure}
    
    \begin{figure}[H]
    \centering
     
    \includegraphics[height=6cm, width= 6cm]{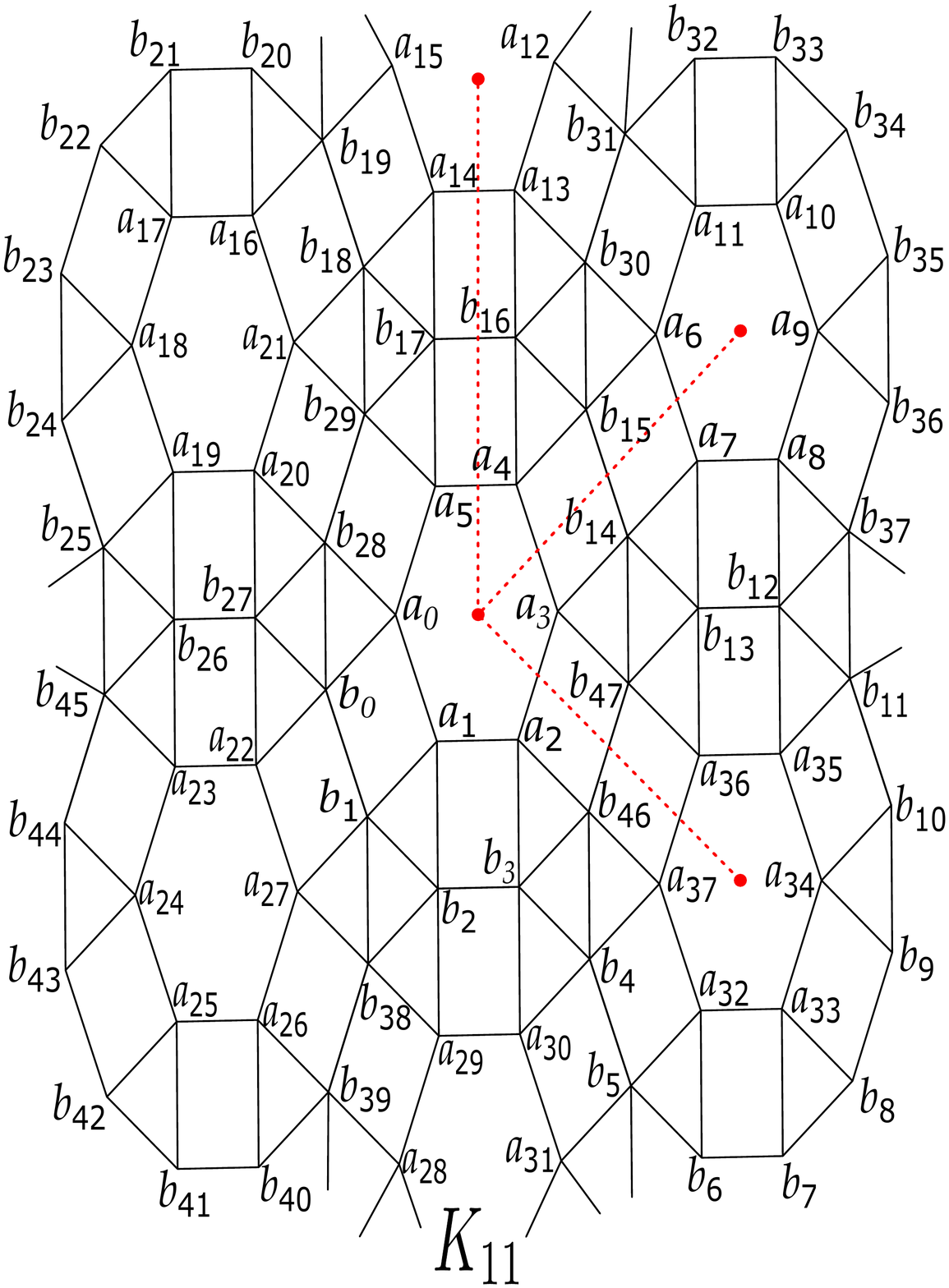}\hspace{5mm}
    \includegraphics[height=6cm, width= 6cm]{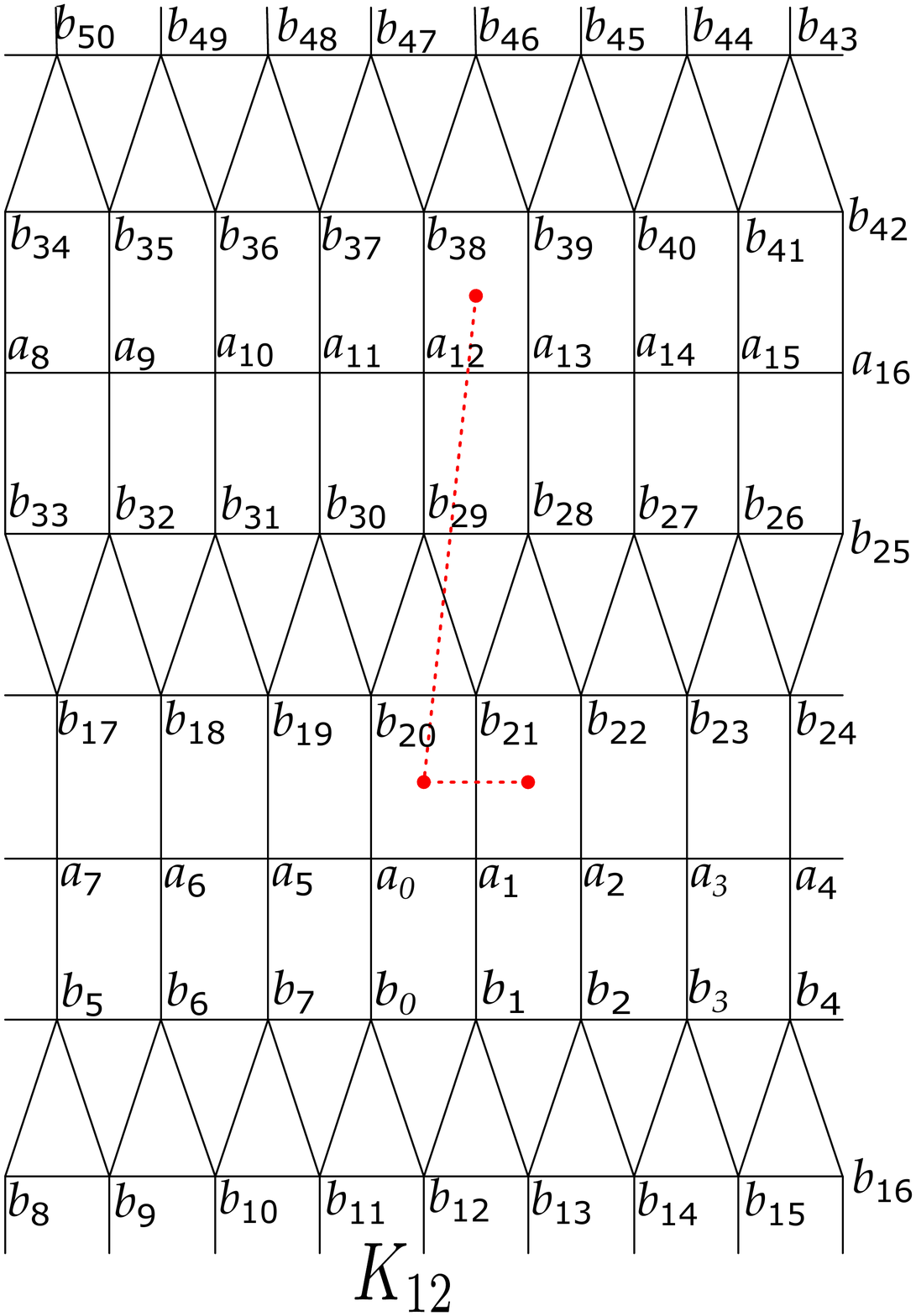}
    
     \end{figure}
     
     \begin{figure}[H]
    \centering
    \includegraphics[height=6cm, width= 6cm]{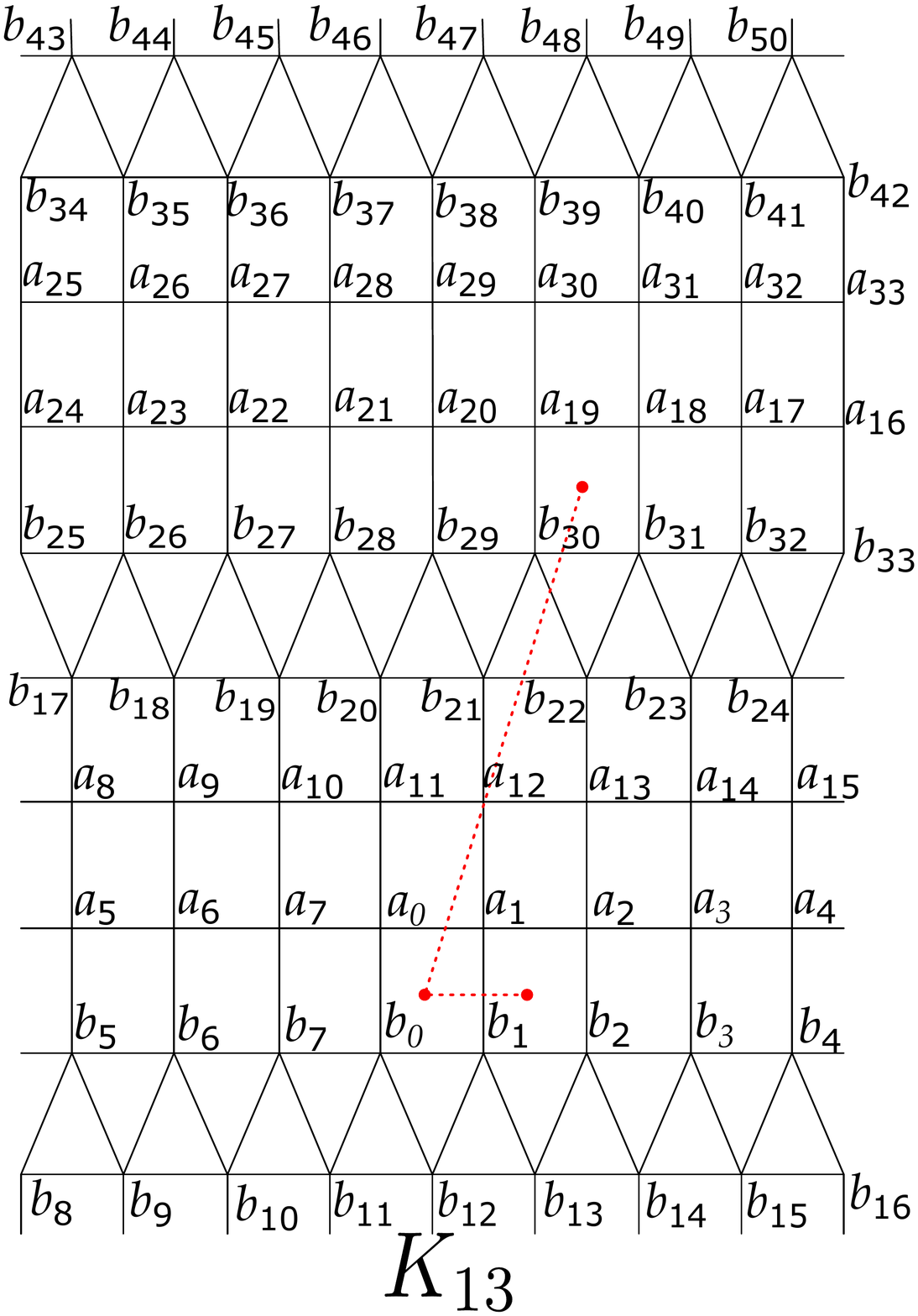}\hspace{5mm}
    \includegraphics[height=6cm, width= 6cm]{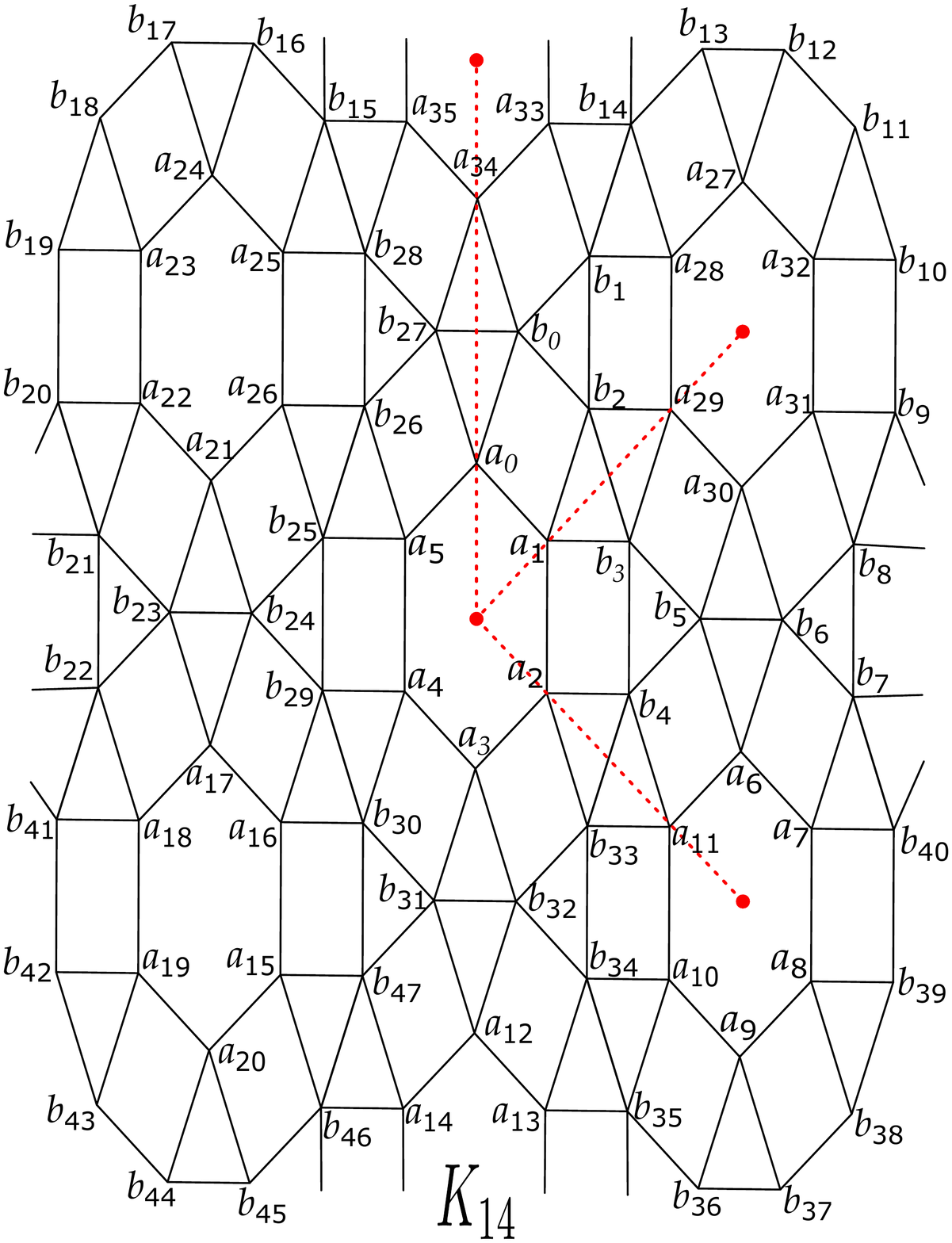}
    
     \vspace{10mm}
     
    \includegraphics[height=6cm, width= 6cm]{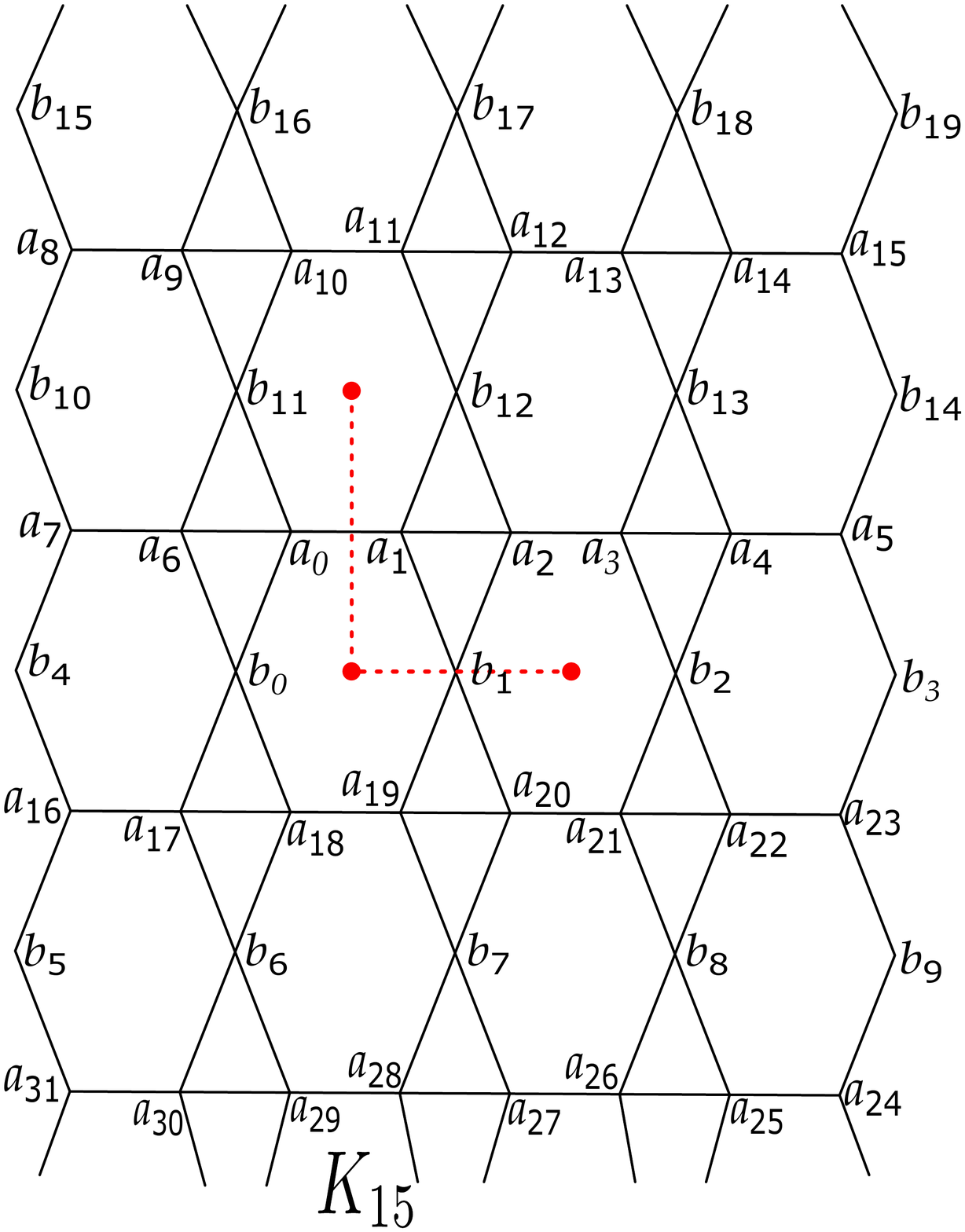}\hspace{5mm}
    \includegraphics[height=6cm, width= 6cm]{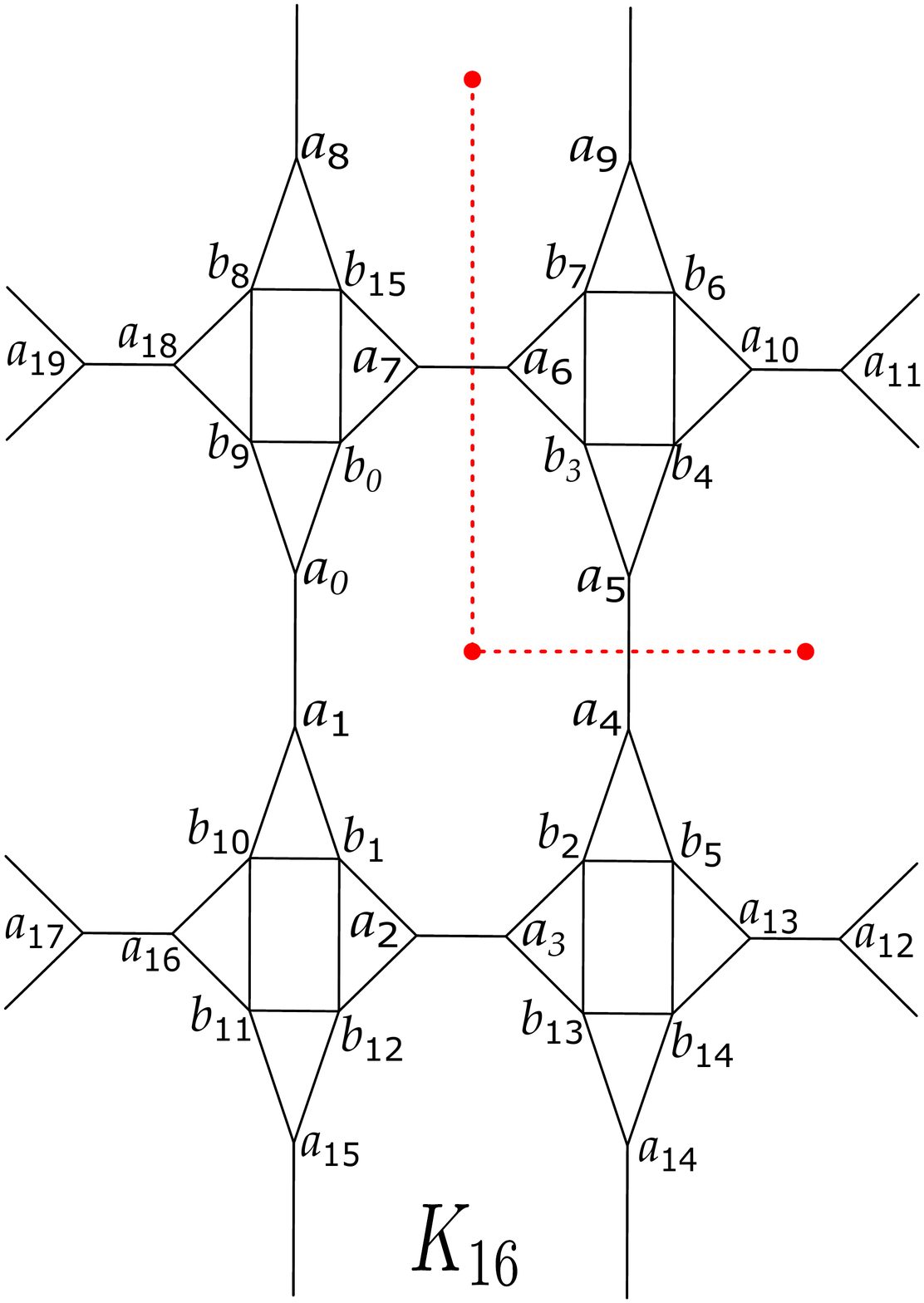}
    
     \vspace{10mm}
     \end{figure}
    
    \begin{figure}[H]
    \centering
    \includegraphics[height=6cm, width= 6cm]{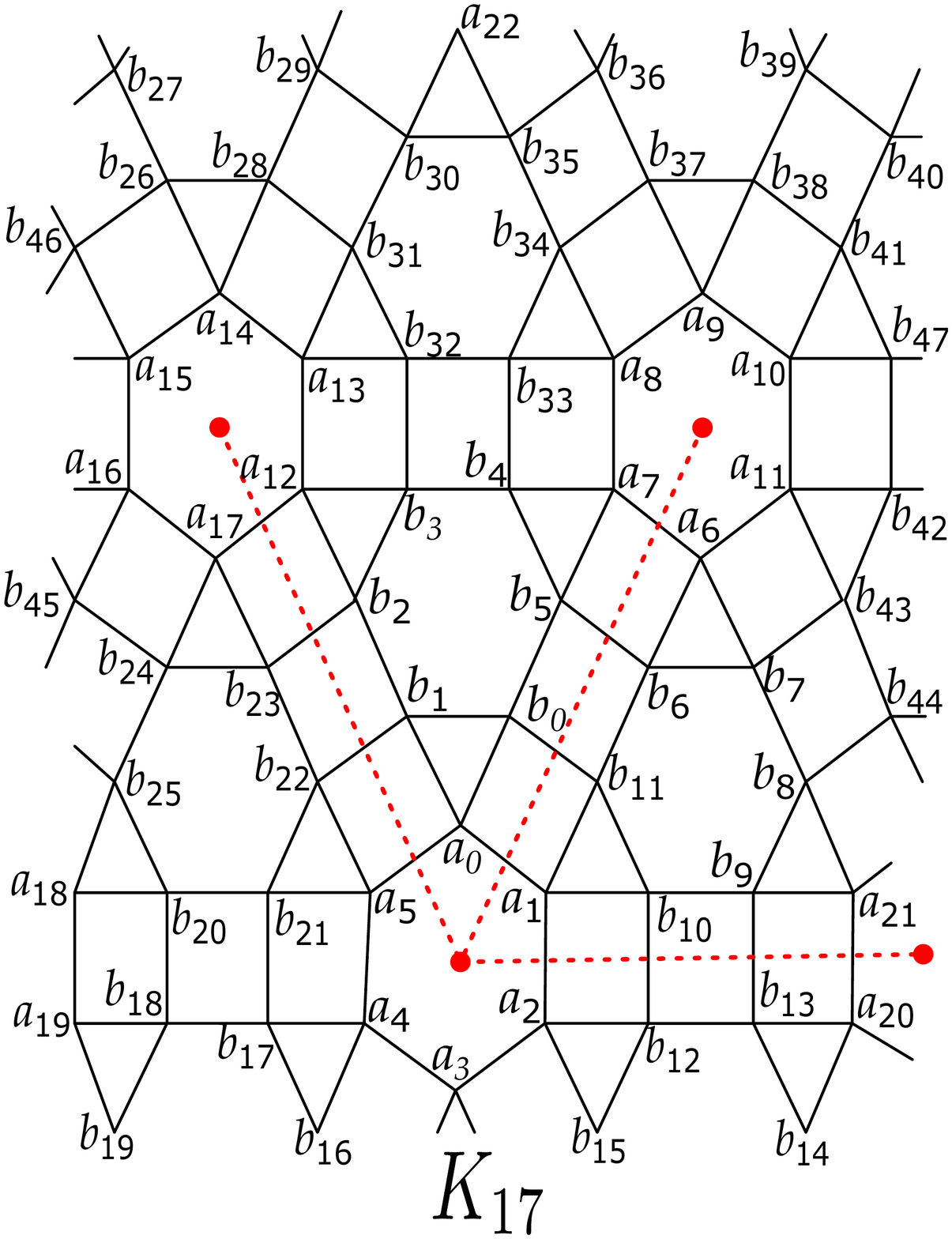}\hspace{5mm}
    \includegraphics[height=6cm, width= 6cm]{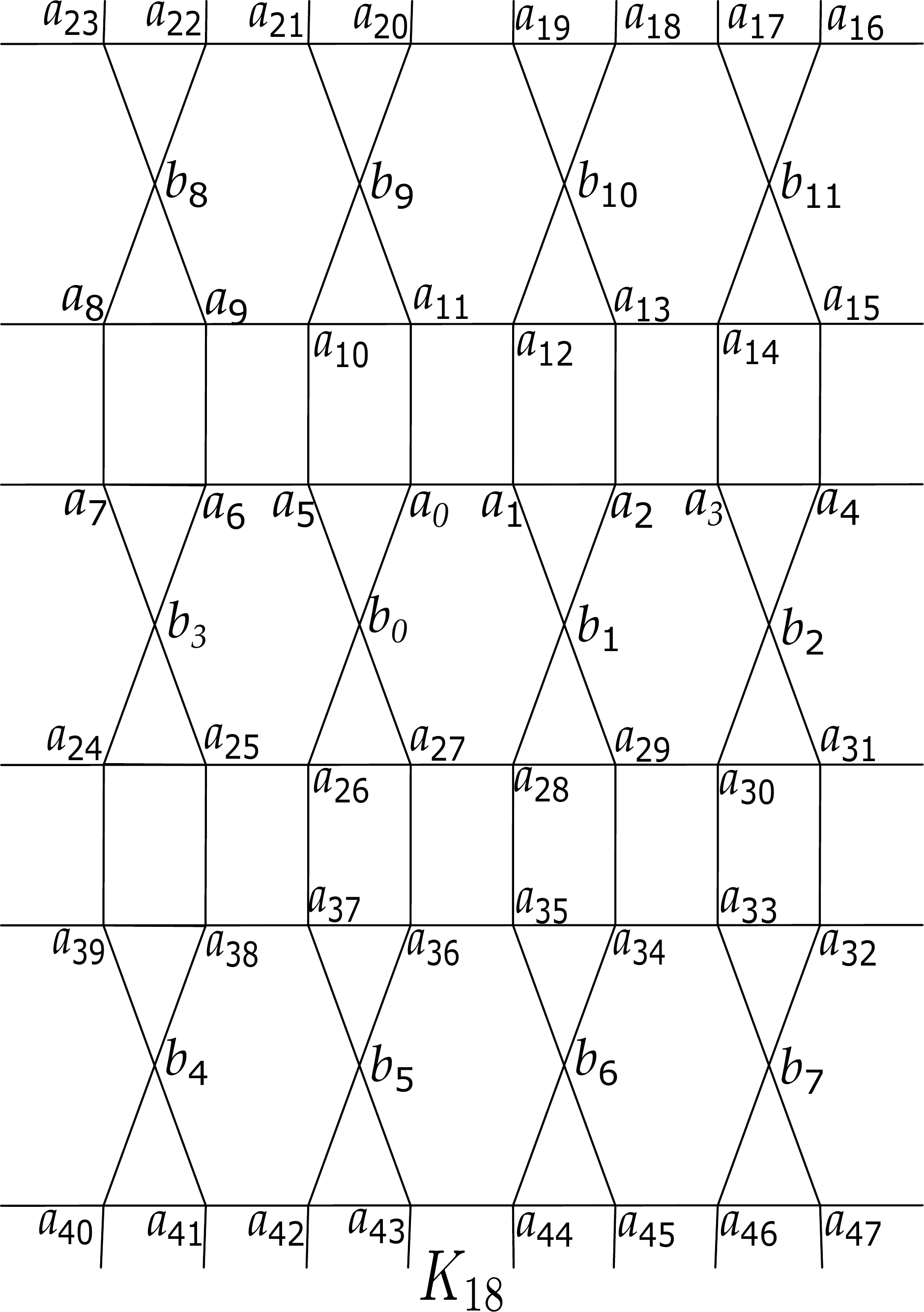}
     \vspace{10mm}
     
     \end{figure}
     
     \begin{figure}[H]
    \centering
    \includegraphics[height=6cm, width= 6cm]{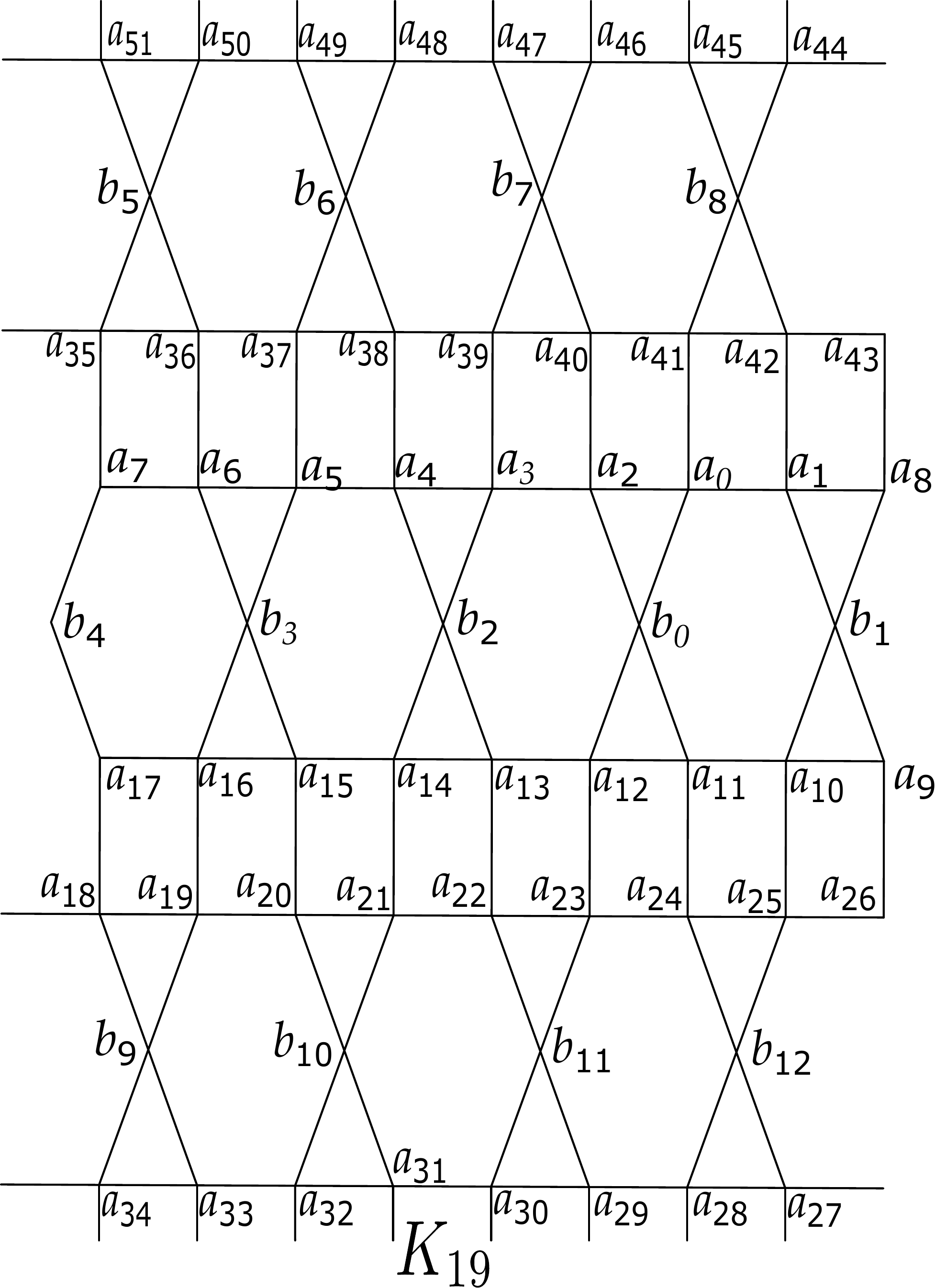}\hspace{5mm}
    \includegraphics[height=6cm, width= 6cm]{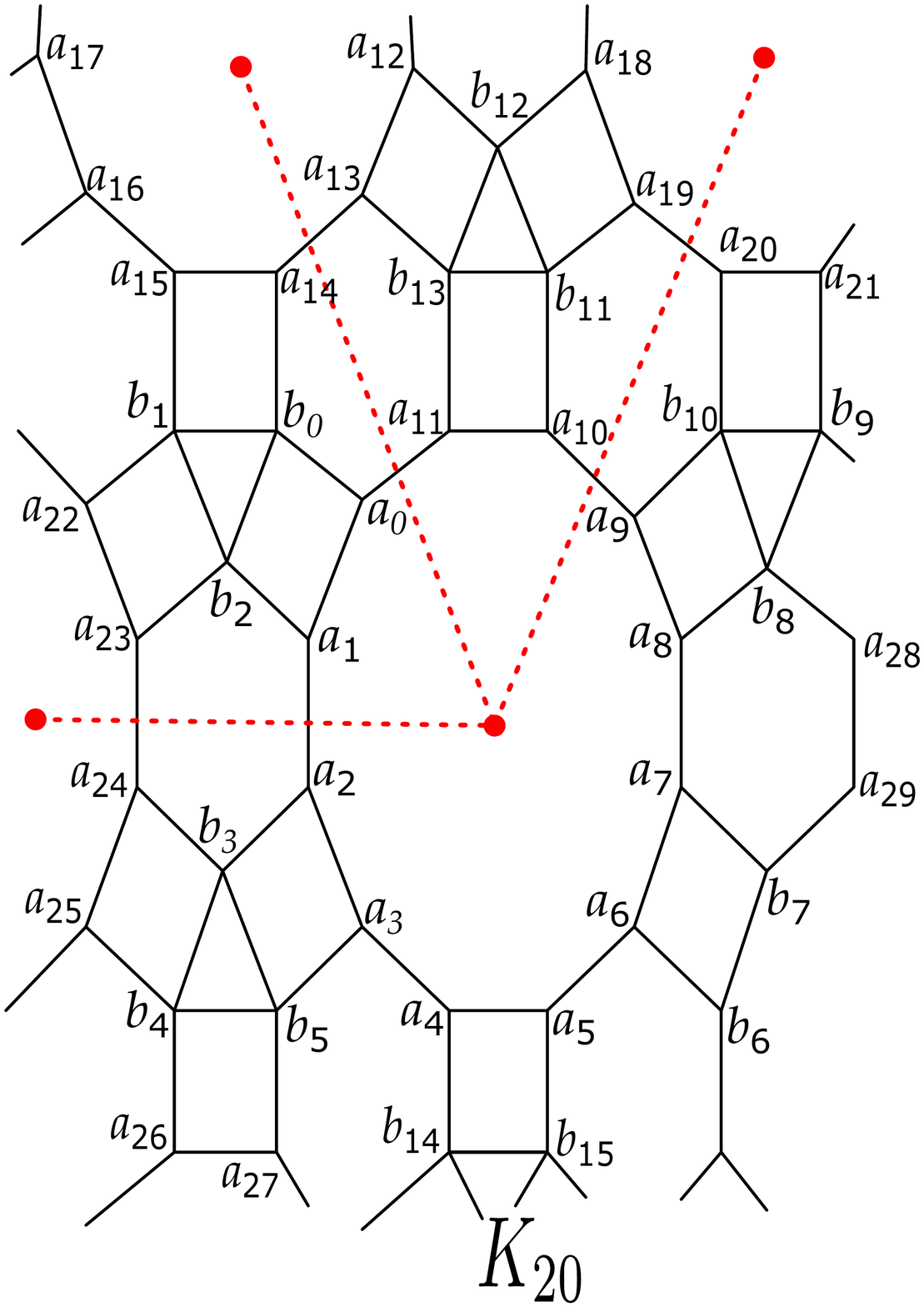}
\end{figure}

\section{Proof }\label{sec:proofs-1}
 Let $K_i$ (given in Section \ref{2uniform}) be of type $A$, where 
 \begin{align*}
   A   \in &\{ [3^{6};3^4,6^1], [3^6;3^3;4^{2}], [3^{6};3^2,4^1,3^1,4^1], [3^{6}; 3^2,4^1,12^1],\\
   &  [3^4,6^1;3^2,6^{2}], [3^3, 4^2;3^2,4^1,3^1,4^1], [3^{3}, 4^2;3^1,4^1,6^1,4^1], [3^3,4^2;4^4], \\
   &  [3^2,4^1,3^1,4^1;3^1,4^1,6^1,4^1], [3^2,6^2; 3^1, 6^1, 3^1, 6^1], [3^1, 4^1, 3^1, 12^1; 3^1, 12^2],\\
   &  [3^1,4^2,6^1; 3^1, 4^1, 6^1, 4^1], [3^1,4^2, 6^1; 3^1, 6^1, 3^1, 6^1], [3^1, 4^1, 6^1, 4^1; 4^1, 6^1, 12^1] \}.  
 \end{align*}
 
Gr\"{u}nbaum and G. C. Shephard \cite{GS1977, GS1981} and Kr\"{o}tenheerdt \cite{Otto1977} have discussed the existence and uniqueness of the $2$-uniform tilings $K_i$, $i =1, 2, \dots, 20$ of the plane.  Thus, we have the following. 

\begin{proposition}\label{prop1}
The $2$-uniform maps $K_{i}$ ($1 \le i \le 20$)  are unique up to isomorphism. 
\end{proposition}

\begin{proof}[Proof of Theorem \ref{theo1}(a)]

\noindent Case 1. Let $X_1$ be a $2$-semiequivelar map on the torus that is the quotient of the plane's $2$-uniform lattice $K_1$ (see in Section \ref{2uniform}).
Let $V_{1} = V(K_1)$ be the vertex set of $K_1$. Let $H_{1}$ be the group of all the translations of $K_1$. So, $H_1 \leq $Aut$(K_1)$.

By assumption, $X_1$ is a $2$-semiequivelar map on the torus and it is quotient of $K_1$, so, we can assume, there is a polyhedral covering map $\eta_{1} : K_1 \to X_1$ where $X_1 = K_1/\Gamma_{1}$  for some fixed element (vertex, edge or face) free subgroup $\Gamma_{1} \le $Aut$(K_1)$. Hence $\Gamma_{1}$
consists of translations and glide reflections. Since $X_1 =K_1/\Gamma_{1}$ is orientable, $\Gamma_{1}$ does not contain any glide reflection. Thus $\Gamma_{1} \leq H_{1}$.

We take the middle point of the line segment joining vertices $a_{0}$ and $a_3$ as the origin $(0,0)$ of $K_1$. Let $A_1 := a_{9} - a_{0}$, $B_1 := a_{14}- a_{0}$ and  $F_1 := a_{20} - a_{0}$ in $K_1$. Clearly, $B_1 = \ell_1 A_1+t_1 F_1$ for some $\ell_1, t_1 \in \mathbb{Z}$. Then $$H_1 := \langle \alpha_1 \colon x\mapsto x+A_1, \beta_1 \colon x\mapsto x+B_1\rangle.$$ Under the action of $H_1$, vertices of $K_1$ form twelve orbits. The orbits are 
\begin{align*}
& O_1 :=\langle a_{0} \rangle, O_2 :=\langle a_{1} \rangle, O_3 :=\langle a_{2} \rangle, 
O_4 :=\langle a_{3} \rangle, O_5 :=\langle a_{4} \rangle, O_6 :=\langle a_{5} \rangle,\\
& O_7 :=\langle b_{18} \rangle, O_8 :=\langle b_{20} \rangle, O_9 :=\langle b_{28} \rangle, O_{10} :=\langle b_{39} \rangle, O_{11} :=\langle b_{37} \rangle, O_{12} :=\langle b_{27} \rangle.
\end{align*}

Let $\rho_1$ be the function obtained by $60$ degrees anticlockwise rotation. Then $\rho_1 \in  Aut(K_1)$ and $\rho_1(A_1) = B_1, \rho_1(B_1) = F_1$. 
Let $G_1 = \langle \alpha_1, \beta_1,  \rho_1\rangle$. Clearly, vertices of $K_1$ form $O_1 :=\langle a_{0} \rangle,  O_2 :=\langle b_{18} \rangle$ $G_1$-orbit. So, $G_1$ acts $2$-uniformly on $K_1$.

Since $\Gamma_{1} \leq H_{1}$, $\Gamma_{1} = \langle \gamma_1 \colon x\mapsto x+C_1, \delta_1 \colon x\mapsto x+D_1\rangle$ where $C_1 = a_1A_1 + b_1B_1$ and $D_1 = c_1A_1 + d_1B_1$, for some $a_1, b_1, c_1, d_1 \in \mathbb{Z}$.

\smallskip

\noindent {\bf Claim 1.}  $L_1 := \langle \alpha_1^m, \beta_1^m  \rangle \le \Gamma_{1}$ for some $m \in \mathbb{Z}$.

\smallskip

Since $K_1/\Gamma_1$ is compact, $C_1$ and $D_1$ are linearly independent. Therefore, there exists $a, b, c, d \in \mathbb{Q}$ such that $A_1 = aC_1 + bD_1$ and $B_1 = cC_1 + dD_1$. Let $m$ be the smallest positive integer such that $ma,mb,mc,md \in \mathbb{Z}$. Then $mA_1 = (ma)C_1 + (mb)D_1, mB_1 = (mc)C_1 + (md)D_1$. Thus, $\alpha_1^m (z) = z + mA_1 = (\gamma_1^{ma} \circ \delta_1^{mb})(z), \beta_1^m(z)= z+mB_1 = (\gamma_1^{mc} \circ \delta_1^{md})(z)$ and hence $\alpha_1^m, \beta_1^m \in \Gamma_1$.
This proves Claim 1.

Since $\Gamma_1$ is abelian, we have $L_1 \unlhd \Gamma_1 \unlhd H_1 \le G_1 \le Aut(K_1)$.

\smallskip

\noindent {\bf Claim 2.} $L_1 \unlhd G_1 $.

\smallskip
For $u, v \in \mathbb{R}$ and $p \in \mathbb{Z}$, $(\rho_1 \circ \alpha_1^p \circ \rho_1^{-1})(uA_1+vB_1) = (\rho_1 \circ \alpha_1^p)(vA_1 - uB_1) = \rho_1((vA_1-uB_1)+pA_1) = \rho_1((p+v)A_1-uB_1)=(p+v)B_1+uA_1 = (uA_1+vB_1)+pB_1 = \alpha_2^p(uA_1+vB_1).$ Thus, $\rho_1 \circ \alpha_1^p \circ \rho_1^{-1} = \alpha_2^p$. Again, for $u, v \in \mathbb{R}$, $\rho_1\circ\alpha_2^p\circ\rho_1^{-1} (uA_1+vB_1) = (\rho_1 \circ \alpha_2^p)(vA_1-uB_1) = \rho_1((vA_1-uB_1)+pB_1) = \rho_1(vA_1+(p-u)B_1) = vB_1-(p-u)A_1 = (uA_1+vB_1)-pA_1 = \alpha_1^{-p}(uA_1+vB_1)$. Thus, $\rho_1 \circ \alpha_2^p \circ \rho_1^{-1} = \alpha_1^{-p}.$ In particular, $\rho_1 \circ \alpha_1^m \circ \rho_1^{-1} = \alpha_2^m$, $\rho_1 \circ \alpha_2^m\circ\rho_1^{-1}=\alpha_1^{-m} \in L_1.$ Since $\alpha_1, \alpha_2$ commute, $L_1 \unlhd G_1 $.

By Claim 2, $L_1$ is a normal subgroup of $G_1$. Therefore, $G_1/L_1$ acts on $Y_1= K_1/L_1$ and $u + L_1 \mapsto u+ \Gamma_1$  gives a covering $\gamma_1 \colon Y_1 \to X_1$.
Since  $\langle a_{0} \rangle, \langle b_{18} \rangle$
 are the $G_1$-orbits, it follows that $O_j/L_1$ for $j=1, 2$ are the $(G_1/L_1)$-orbits. Clearly, $G_1/L_1 \le Aut(Y_1)$. It follows that the number of Aut$(Y_1)$-orbits of vertices is $2$. This completes Theorem \ref{theo1} when $X_1$ is associated to $K_1$. 
 
\medskip
 
 \noindent Case 2. Let $X_2$ be a $2$-semiequivelar map on the torus and $X_2 = K_2/\Gamma_{2}$ (see $K_2$ in Section \ref{2uniform}) for some fixed element (vertex, edge or face) free subgroup $\Gamma_{2} \le $Aut$(K_2)$.
 
Let $V_{2} = V(K_2)$ be the vertex set of $K_2$. Let $H_{2}$ be the group of all the translations of $K_2$. So, $H_2 \leq $Aut$(K_2)$. Since $X_2 = K_2/\Gamma_{2}$, $\Gamma_{2}$
consists of translations and glide reflections. Since $X_2 =
K_2/\Gamma_{2}$ is orientable, $\Gamma_{2}$ does not contain any glide reflection. Thus $\Gamma_{2} \leq H_{2}$.

 We take the middle point of the line segment joining vertices $a_{0}$ and $a_3$ as the origin $(0,0)$ of $K_2$. Let $A_2 := a_{6} - a_0$, $B_2 := a_{18} - a_0$ and $F_2 := a_{24} - a_{0}$ $\in \mathbb{R}^2$. Similarly as in Case 1, define $H_2, G_2, L_2$. The result follows in this case by similar argument as in Case 1. 
 
 \medskip

 \noindent Case 3.  Let $X_5$ be a $2$-semiequivelar map on the torus and $X_5 = K_5/\Gamma_{5}$ (see $K_5$ in Section \ref{2uniform}) for some fixed element (vertex, edge or face) free subgroup $\Gamma_{5} \le $Aut$(K_5)$.

We take the middle point of the line segment joining vertices $b_{0}$ and $b_{3}$ as the origin $(0,0)$ of $K_5$ (see in Section \ref{2uniform}). Let  $A_5 := a_1 - a_0$, $B_5 := a_{7} - a_0$ and $F_5 := a_{6} - a_0$ $\in \mathbb{R}^2$. Similarly as above in Case 1, define $H_5, G_5, L_5$. The result follows in this case by exactly same argument as in in Case 1. 

\medskip
 
\noindent Case 4.  Let $X_6$ be a $2$-semiequivelar map on the torus and $X_6 = K_6/\Gamma_{6}$ (see $K_6$ in Section \ref{2uniform}) for some fixed element (vertex, edge or face) free subgroup $\Gamma_{6} \le $Aut$(K_6)$.

We take the middle point of the line segment joining vertices $a_{0}$ and $a_{6}$ as the origin $(0,0)$ of $K_6$. Let  $A_6 := a_{27} - a_1$, $B_6 := a_{32} - a_{0}$ and $F_6 := a_{37} - a_{11}$ $\in \mathbb{R}^2$. Similarly as in Case 1, define $H_6, G_6, L_6$. The result follows in this case by similar argument as in Case 1. 

\medskip
 
\noindent Case 5.  Let $X_7$ be a $2$-semiequivelar map on the torus and $X_7 = K_7/\Gamma_{7}$ (see $K_7$ in Section \ref{2uniform}) for some fixed element (vertex, edge or face) free subgroup $\Gamma_{7} \le $Aut$(K_7)$.

We take the middle point of the line segment joining vertices $b_{1}$ and $b_{4}$ as the origin $(0,0)$ of $K_7$. Let  $A_7 := a_{9} - a_0$, $B_7 := a_{2} - a_{0}$ and $F_7 := a_{1} - a_{0}$ $\in \mathbb{R}^2$. Similarly as above in Case 1, define $H_7, G_7, L_7$. The result follows in this case by similar argument as in Case 1. 

\medskip



 
\noindent Case 6. Let $X_{11}$ be a $2$-semiequivelar map on the torus and $X_{11} = K_{11}/\Gamma_{11}$ (see $K_{11}$ in Section \ref{2uniform}) for some fixed element (vertex, edge or face) free subgroup $\Gamma_{11} \le $Aut$(K_{11})$.

We take origin $(0,0)$ is the middle point of the line segment joining vertices $a_{0}$ and $a_{3}$ of $K_{11}$. Let  $A_{11} := a_{37} - a_0$, $B_{11} := a_6 - a_{0}$ and $F_{11} := a_{15} - a_{0} \in \mathbb{R}^2$. Similarly as in Case 1, define $H_{11}, G_{11}, L_{11}$. The result follows in this case by similar argument as in Case 1.

\medskip
 
\noindent Case 7. Let $X_{14}$ be a $2$-semiequivelar map on the torus and $X_{14} = K_{14}/\Gamma_{14}$ (see $K_{14}$ in Section \ref{2uniform}) for some fixed element (vertex, edge or face) free subgroup $\Gamma_{14} \le $Aut$(K_{14})$.

We take origin $(0,0)$ is the middle point of the line segment joining vertices $a_{0}$ and $a_{3}$ of $K_{14}$. Let  $A_{14} := a_{10} - a_4$, $B_{14} := a_{29} - a_{4}$ and $F_{14} := a_{35} - a_{4} \in \mathbb{R}^2$. Similarly as in Case 1, define $H_{14}, G_{14}, L_{14}$. The result follows in this case by similar argument as in Case 1.

\medskip

\noindent Case 8. Let $X_{16}$ be a $2$-semiequivelar map on the torus and $X_{16} = K_{16}/\Gamma_{16}$ (see $K_{16}$ in Section \ref{2uniform}) for some fixed element (vertex, edge or face) free subgroup $\Gamma_{16} \le $Aut$(K_{16})$.

We take origin $(0,0)$ is the middle point of the line segment joining vertices $a_{0}$ and $a_{4}$ of $K_{16}$. Let  $A_{16} := a_{5} - a_0$ and $B_{16} := a_{8} - a_1 \in \mathbb{R}^2$. Similarly as above in Case 1, define $H_{16}, G_{16}, L_{16}$. The result follows in this case by similar argument as in Case 1.

\medskip
 
\noindent Case 9. Let $X_{17}$ be a $2$-semiequivelar map on the torus and $X_{17} = K_{17}/\Gamma_{17}$ (see $K_{17}$ in Section \ref{2uniform}) for some fixed element (vertex, edge or face) free subgroup $\Gamma_{17} \le $Aut$(K_{17})$.

We take origin $(0,0)$ is the middle point of the line segment joining vertices $a_0$ and $a_3$ of $K_{17}$. Let  $A_{17} := a_{20} - a_4$, $B_{17} := a_7 -a_4$ and $F_{17} :=a_{17} - a_3 \in \mathbb{R}^2$. Similarly as in Case 1, define $H_{17}, G_{17}, L_{17}$. The result follows in this case by similar argument as in Case 1.

\medskip
 


 


\medskip
 
\noindent Case 10. Let $X_{20}$ be a $2$-semiequivelar map on the torus and $X_{20} = K_{20}/\Gamma_{20}$ (see $K_{20}$ in Section \ref{2uniform}) for some fixed element (vertex, edge or face) free subgroup $\Gamma_{20} \le $Aut$(K_{20})$.

We take origin $(0,0)$ is the middle point of the line segment joining vertices $a_{0}$ and $a_{6}$ of $K_{20}$. Let  $A_{20} := a_{19} - a_3$, $B_{20} := a_{14} -a_5$ and $F_{20} := a_{23} - a_8 \in \mathbb{R}^2$. Similarly as in Case 1, define $H_{20}, G_{20}, L_{20}$. The result follows in this case by similar argument as in Case 1.

\medskip

\noindent Case 11. Let $K = K_3, K_4, K_8, K_{12}, K_{13}$ or $K_{15}$ (see $K_i$ for $i = 3, 4, 8, 12, 13, 15$ in Section \ref{2uniform}). Let $X$ be a $2$-semiequivelar map on the torus and $X = K/\Gamma_i$  for some fixed element (vertex, edge or face) free subgroup $\Gamma_{i} \le $Aut$(K_i)$. Let the vertices of $X$ form $m_i$ Aut$(X_i)$-orbits. Then, from \cite{MDD2020}, we know $m_i =2$.  So, the identity map $X \to X$ is a covering map. 
\end{proof}

\begin{proof}[Proof of Theorem \ref{theo1}(b)]  Let $X_9$ be a $2$-semiequivelar map on the torus and $X_9 = K_9/\Gamma_{9}$ (see $K_9$ in Sec. \ref{2uniform}) for some fixed element (vertex, edge or face) free subgroup $\Gamma_{9} \le $Aut$(K_9)$.  Since $X_9 = K_9/\Gamma_{9}$ is orientable, $\Gamma_{9}$ does not contain any glide reflection. Let $H_{9}$ be the group of all the translations of $K_9$. Thus $\Gamma_{9} \leq H_{9} \leq$ Aut$(K_9)$. In $K_9$, observe that the elements in the orbit $O(b_{34})$ maps to the elements in the orbit $O(b_{33})$ only under some glide reflection symmetry or reflection symmetry about a line. Here, both are not fixed element free. Since $\Gamma_{9}$ does not contain these two symmetries, the number of vertex Aut$(X_9)$-orbit is at least three.  Hence,  there does not exist any covering $\gamma : M \to X_9$ where $M$ is a $2$-uniform toroidal map. Similarly, if $X_{i}$ is a $2$-semiequivelar map on the torus and $X_i = K_i/\Gamma_{i}$ for $i=10, 18, 19$, the number of vertex Aut$(X_i)$-orbit is at least three. Hence, in these cases also, covering map  $\gamma : M \to X_i$ where $M$ is a $2$-uniform toroidal map does not exist. 
\end{proof}

\begin{proof}[Proof of Theorem \ref{theo1}(c)] Let $X$ be a $2$-semiequivelar map on the torus that is not the quotient of the plane's $2$-uniform lattice $K_i$ for $1 \le i \le 20$ (see in Section \ref{2uniform}). We know that the plane is the universal cover of the torus. So, there exists a tiling $Y$ and a covering group $\Gamma$ such that $X = Y /\Gamma$. Since $Y \not\cong K_i$  $\forall ~i$, $Y$ is not $2$-uniform and the number of Aut$(Y)$-orbit is at least three. So, for any fixed element free subgroup $H$ of Aut$(Y)$, the number of $H$-orbit is at least three. Hence,  there does not exist any covering $\beta : M \to X$ where $M$ is a $2$-uniform toroidal map.  This completes the part (c).
\end{proof}




{\small

}

\end{document}